\documentclass[11pt]{amsart}
\usepackage{amsthm}
\usepackage{amssymb}
\usepackage{amsmath}
\usepackage{graphicx}

\usepackage[margin=1in]{geometry}
\usepackage[utf8]{inputenc} 

\theoremstyle{plain}
\newtheorem{thm}{Theorem}

\newtheorem{claim}[thm]{Claim}
\newtheorem{lem}[thm]{Lemma}

\theoremstyle{definition} 

\theoremstyle{remark} \newtheorem{rem}{Remark}

\newcommand{\mymid}{\,:\,}
\newcommand{\bigmid}{\,\middle\vert\,}
\newcommand{\Fill}{\operatorname{fill}}
\newcommand{\Flat}{\operatorname{flat}}
\newcommand{\Fold}{\operatorname{fold}}

\title{Vertex isoperimetry and independent set stability\\ for tensor
powers of cliques}
\author{Joshua Brakensiek}
\thanks{Department of Mathematical Sciences, Carnegie Mellon
  University, Pittsburgh, PA 15213, USA. Email: {\tt jbrakens at
    andrew.cmu.edu}. This work was partially supported by REU
  supplements to NSF CCF-1526092 and CCF-1422045}

\begin{document}

\begin{abstract}

  The tensor power of the clique on $t$ vertices (denoted by $K_t^n$)
  is the graph on vertex set $\{1, \hdots, t\}^n$ such that two
  vertices $x, y \in \{1, \hdots, t\}^n$ are connected if and only if
  $x_i \neq y_i$ for all $i \in \{1, \hdots, n\}$. Let the density of
  a subset $S$ of $K_t^n$ to be $\mu(S) := \frac{|S|}{t^n}$, and let
  the vertex boundary of a set $S$ to be
  vertices which are incident to some vertex of $S$, perhaps including
  points of $S$.  We investigate two similar problems on such graphs.

  First, we study the vertex isoperimetry problem. Given a density
  $\nu \in [0, 1]$ what is the smallest possible density of the vertex
  boundary of a subset of $K_t^n$ of density $\nu$? Let $\Phi_t(\nu)$
  be the infimum of these minimum densities as $n \to \infty$.  We find
  a recursive relation allows one to compute $\Phi_t(\nu)$ in time
  polynomial to the number of desired bits of precision.

  Second, we study given an independent set $I \subseteq K_t^n$ of
  density $\mu(I) = \frac{1}{t}(1-\epsilon)$, how close it is to a
  maximum-sized independent set $J$ of density $\frac{1}{t}$. We show
  that this deviation (measured by $\mu(I \setminus J)$) is at most
  $4\epsilon^{\frac{\log t}{\log t - \log(t-1)}}$ as long as $\epsilon
  < 1 - \frac{3}{t} + \frac{2}{t^2}$. This substantially improves on
  results of Alon, Dinur, Friedgut, and Sudakov (2004) and
  Ghandehari and Hatami (2008) which had an $O(\epsilon)$ upper
  bound. We also show the exponent $\frac{\log t}{\log t - \log(t-1)}$
  is optimal assuming $n$ tending to infinity and $\epsilon$ tending
  to $0$. The methods have similarity
  to recent work by Ellis, Keller, and Lifshitz (2016) in the context
  of Kneser graphs and other settings.

  The author hopes that these results have potential applications in
  hardness of approximation, particularly in approximate graph
  coloring and independent set problems. 
\end{abstract}

\maketitle

\section{Introduction}\label{sec:intro}

\subsection{Vertex isoperimetry}\label{subsec:vertex-iso}
For any undirected graph $G = (V_G, E_G)$ and $S \subseteq V_G$, we
define the \emph{vertex boundary} of $S$ to be \[
\partial S := \{x \in V_G \mymid \text{exists $y \in S$ such that $\{x, y\}
\in E_G$}\}.
\]
Furthermore, we define the \emph{density} of $S$ to be
\[
\mu(S) := \frac{|S|}{|V_G|}.
\]
The relationship between $\mu(S)$ and $\mu(\partial S)$, particularly
when $\mu(S)$ is sufficiently small (typically at most $1/2$). Is
known as a \emph{vertex isoperimetric inequality}. Such relationships
are captured by the \emph{isoperimetric parameter} (or
\emph{isoperimetric profile}) of a graph
\[
\Phi(G, \nu) = \inf\{\mu(\partial S) \mymid \mu(S)\ge \nu\}
\]
Proving such inequalities for various graphs is a frequent topic in
the literature (e.g., \cite{Bobkov2000, Christofides2013}). Typically
such works focus on a \emph{linear} or near-linear relationship
between $\mu(\partial S)$ and $\mu(S)$, known as the
\emph{isoperimetric constant}.
\begin{align}
h(G) &= \inf\left\{\frac{\mu(\partial S)}{\mu(S)} \bigmid S \subset
  V_G, \mu(S) \in (0, 1/2] \right\}.
\end{align}
In this paper, we study graphs for which there is an
order-of-magnitude difference between $\mu(S)$ and $\mu(\partial S)$,
when $\mu(S)$ is sufficiently small. For example, if $\mu(\partial S)
\ge \sqrt{\mu(\partial S)}$ for all $S$, we would like to say that $G$
expands by a power of $2$. Such `hyper-expansion' can be captured by
what we coin as the \emph{isoperimetric exponent}. For all $\epsilon >
0$ consider.
\begin{align}
\eta(G, \epsilon) &= \inf\left\{\frac{\log \mu(S)}{\log \mu(\partial S)} 
  \bigmid S \subset V_G, \mu(\partial S) \in (0, \epsilon)\right\}
\end{align}
where $\log$ is the natural logarithm. In other words, for every
subset $S$ of $G$ of density $\delta$, the boundary of $S$ has density
at least $\delta^{1/\eta(G, \epsilon)}.$ The larger the parameter
$\eta(G)$ is, the more `expansive' the graph is. It is easy to see
that $\eta(G, \epsilon)$ is in general a decreasing function of
$\epsilon$. As we often work with large subsets of our graph, we let
$\eta(G) := \eta(G, 1)$.

In this paper, we study the isoperimetric profile of the tensor powers
of cliques. For undirected graphs $G = (V_G, E_G), H = (V_H, E_H)$, we
define the \emph{tensor product} $G \otimes H$ to be the undirected
graph on vertex set $V_1 \times V_2$ such that an edge connects $(u_1,
v_1)$ and $(u_2, v_2)$ if and only if $\{u_1, u_2\} \in E_G$, and
$\{v_1, v_2\} \in E_H.$ Note that up to isomorphism, the tensor
product is both commutative and associative. We then denote $\otimes^n
G$ to be the tensor product of $n$ copies of $G$. Since this is the
only graph product discussed in this article, we shorten this to
$G^n$. In this article, we focus on the case that $G = K_t$, where
$K_t$ is the complete graph on $t \ge 3$ vertices. It turns out for
such graphs that for all $\epsilon > \frac{1}{t^n}$, $\eta(G) =
\eta(G, \epsilon)$.

In particular, we shall compute the following.
\begin{thm}\label{thm:isoperimetry-tensor}
For all $t \ge 3$ and all positive integers $n$,
\begin{align}
\eta(K_t^n) = \eta(K_t) = \frac{\log t}{\log t - \log (t-1)} = t \log t + \Theta(\log t).
\end{align}
\end{thm}

In addition to this high-level structure, we give a more-fine-tuned
analysis of the behavior of $\Phi_t(\eta) := \inf_{n \ge 1}
\Phi(K_t^n, \eta)$. (See Theorem~\ref{lem:bound-D-exact}.)

\subsection{Independent set stability}\label{subsec:ind-stab}
With these vertex isoperimetric inequalities, we apply them
to the understanding the structure of near-maximum independent sets of
graphs. Such results are known as \emph{stability} results.

Such results are not just of interest within combinatorics, a better
understanding of independent set stability of certain graphs, such as
$K_t^n$, have resulted in advances in hardness of approximation,
particularly in construct dictatorship tests for approximate graph
coloring and independent set problems (e.g., \cite{Alon, Dinur2008,
  Brakensiek2016b}). In fact the investigation which led to the
results in this paper was inspired by the pursuit of such results.

A landmark result of this form due to \cite{Alon} is as follows.

\begin{thm}[\cite{Alon}]\label{thm:ADFS} For all $t \ge 3$
  there exist $C_t$ with the following property.  For any positive
  integer $n$, Let $I \subset\ K_t^n$ be an independent set
  such that $\epsilon = 1 - t\mu(I)$, then there exists an
  independent set $J \subset K_t^n$ of maximum size ($\mu(J) =
  1/t$) such that $\mu(I \Delta J) \le C_t\epsilon$, where $S \Delta T = (S
  \setminus T) \cup (T \setminus S)$.
\end{thm}

In other words, independent sets of near-maximum size are similar in
structure to the maximum independent sets. Note that if $J$ is an
independent set of maximum size, then for some $i \in [n]$ and $j \in
[t]$, we have that
\[
J = [t]^{i-1} \times \{j\} \times [t]^{n - i}.
\]
This is a well-known result due to \cite{Greenwell1974} (see
\cite{Alon2004} for a proof using Fourier analysis).

Ghandehari and Hatami improved this result (Theorem 1 of
\cite{Ghandehari2008}) to show that if $t \ge 20$ and $\epsilon \le
10^{-9}$ then $C_t$ can be replaced with $40/t$. Both results were
proven using Fourier analysis.

We improve upon this result in two steps. First, with an application
of Theorem~\ref{thm:isoperimetry-tensor} we improve
Theorem~\ref{thm:ADFS} in a black-box matter to obtain.

\begin{thm}\label{thm:tensor-stability}
  For all $t \ge 3$, there exists $\epsilon_t > 0$ with the following property.
   For any positive
  integer $n$, Let $I \subset\ K_t^n$ be an independent set
  such that $\epsilon = 1 - t\mu(I) < \epsilon_t$, then there exists an
  independent set $J \subset K_t^n$ of maximum size ($\mu(J) =
  1/t$) such that \begin{align}\mu(I \setminus J) \le 4\epsilon^{\eta(K_t)} =
  4\epsilon^{\log t/(\log t - \log(t-1))}.\label{eq:stab-tensor-clique}
\end{align}
\end{thm}

\begin{rem}
  Since $\mu (I \setminus J) \le 4\epsilon^{\eta(K_t)}$,
  \[
  \mu(I \Delta J) = \mu(I \setminus J) + \mu(J \setminus I) = \mu(J) -
  \mu(I) + 2\mu(I \setminus J) = \frac{\epsilon}{t} + 4\epsilon^{\eta(K_t)},
  \]
  so our result gives the optimal first-order structure for
  Theorem~\ref{thm:ADFS} assuming $\epsilon$ is sufficiently
  small. Furthermore, in Appendix~\ref{app:opt}, we give
  examples of independent sets of $K_t^n$ with arbitrarily small
  density (assuming $n \to \infty$) for which the exponent $\eta(K_t)$
  is optimal.
\end{rem}

Next, using a purely combinatorial argument we pin down a precise value for $\epsilon_t$.

\begin{thm}\label{thm:tensor-stability-better}
  In Theorem \ref{thm:tensor-stability}, for all $t \ge 3$, one may
  set $\epsilon_t = 1 - \frac{3}{t} + \frac{2}{t^2}$. In other words,
  the theorem applies for all independent sets $I$ such that $\mu(I)
  > \frac{3t - 2}{t^3}$. 
\end{thm}

The choice of $\epsilon_t$ is not arbitrary, it corresponds to the
density of the following independent set.

\[
I = \{(1, 1, a), (1, a, 1), (a, 1, 1) \mymid a \in [t]\} \times [t]^{n-3}. 
\]

Note that $\mu(I) = \frac{3t-2}{t^3}.$ This set represents a phase
transition in the independent sets from `dictators' to `juntas,' as
the $I$ constructed above is equally influenced by $3$ coordinates
(where `influence' is in the sense of \cite{Alon}). Such phase
transitions have been studied in the literature \cite{Dinur2008}, but
this may be the first work to highlight the exact transition point.

Additionally, to the best of the author's knowledge,
this is the first known purely combinatorial proof of
Theorem~\ref{thm:ADFS}.

\subsubsection{Related work}\label{subsubsec:kneser}

Such stability results for independent sets have also been studied for
Kneser graphs. A result similar to that of Theorem~\ref{thm:ADFS} was proved by
\cite{Friedgut2008}. Numerous other works in the literature
\cite{Dinur2009a, Dinur2005, Balogh2008, Keevash2008,
  Keevash2010,Filmus2016,Filmus2016a} prove generalized stability
results for Kneser graphs or other structures related to intersecting
families. 

A result which also finds a ``tight'' super constant exponent $\eta >
1$ for the independent set stability is proved in some very recent
work
\cite{Ellis2016,Ellis2016a,Ellis2016b,Keller2016,Keller2016a,Ellis2017}
on Kneser graphs and related structures. (See also \cite{Ellis2017a}
and Proposition 4.3 of \cite{Filmus2016b}.) The techniques have
high-level similarity to the ones adopted here:\footnote{The author
  became aware of these similar proofs only after writing major portions
  of the manuscript.}  particularly in their use of compressions to
prove a isoperimetric inequality which they then bootstrap to a
combinatorial independent set stability result.

\subsection{Paper organization}

In Section~\ref{sec:isoperimetry} we prove the claimed vertex
isoperimetric inequalities. In Section~\ref{sec:stab}, we prove the
stability results for near-maximum independent sets in $K_t^n$.
Appendix~\ref{app:ineq} proves some algebraic inequalities omitted
from the main text. Appendix~\ref{app:profile} proves
Theorem~\ref{lem:bound-D-exact}, which gives a refined understanding
the isoperimetric profile of Kneser graphs. Appendix~\ref{app:opt}
shows that the exponent of $\eta(t)$ in
Theorems~\ref{thm:tensor-stability} and
\ref{thm:tensor-stability-better} is optimal.

\section{Vertex isoperimetric Inequalities} \label{sec:isoperimetry}

In this section, we proceed to prove the isoperimetry results claimed
in Section~\ref{subsec:vertex-iso}.

Identify the vertex set of $K_t^n$ with $[t]^n$. Two vertices of $x, y
\in [t]^n$ are connected in $K_t^n$ if and only if $x_i \neq y_i$ for
all $i \in [n].$ Denote $y_{\neg i} := (y_1, \hdots, y_{i-1}, y_{i+1}, \hdots,
y_n)$. We often write $y$ as $(y_i, y_{\neg i})$ when it is clear from
context which coordinate is being inserted.

\subsection{Compressions}
A useful tool in our study will be the operation of the well-known
technique of \textit{compressions} (e.g., \cite{Sauer1972,
  Shelah1972}). Although compressions are not strictly necessary to
prove Theorem~\ref{thm:isoperimetry-tensor}, they are essential in the
proof of stronger isoperimetry results as well as
Theorem~\ref{thm:tensor-stability-better}, so we introduce the
machinery now.

For $S \subseteq [t]^n$ be a
subset, define the \emph{compression of $S$ in coordinate $i$} to be
  \begin{align}c_i(S) = \left\{x \in \mathbb [t]^n \mymid x_i \le |\{y\in S \mymid
    y_{\neg i} = x_{\neg i}\}|\right\}.\end{align}

Informally, we `shift' each element of $S$ to be as small as possible
in the $i$th direction. Note that $\mu(c_i(S)) = \mu(S)$ for all $S
\subseteq \mathbb [t]^n$. It is easy to see that $c_i$ is
\emph{nilpotent}: $c_i(c_i(S)) = c_i(S)$ for all $S \subseteq [t]^n$
and $i \in [n]$.

We say that a set $S$ is \textit{compressed} if $c_i(S) = S$ for all
$i \in [n]$. Equivalently, for all $x \in S$ there is no $y \in [t]^n
\setminus S$ such that $x_i \le y_i$ for all $i \in [n]$.

\begin{rem}\label{rem:compressed} Note that every time a compression $c_i$ is applied, the
quantity
\[
\Sigma(S) := \sum_{x \in S} \sum_{j \in [n]} x_j
\]
decreases or stays the same (in which case $c_i(S) = S$). Thus, since
$\Sigma(S)$ is always positive, there must exist a finite sequence of
compressions which can be applied to $S$ to make the set compressed.
\end{rem}

Now we show that compressions respect independent
sets of $K_t^n$. This result is not needed until
Section~\ref{sec:stab}, but the proof does give intuition for how the
compressions work.

\begin{claim}\label{claim:c-ind} For all $i \in [n]$ and all $I \subset
  \mathbb [t]^n$ independent set of $K_t^n$, $c_i(I)$ is also an
  independent set of $K_t^n$.
\end{claim}
\begin{proof} Assume not, then there exist $x, y \in c_i(I)$
  such that $\{x, y\}$ is an edge. In particular, since $x_i \neq y_i$, we must have that
  $x_i \neq 1$ or $y_i \neq 1$. Assume without loss of generality that
  $y_i \neq 1$. Then, by definition of $c_i(I)$, there must be $z :=
  (1, y_{\neg i}) \in c_i(I)$. Since $x, y, z \in c_i(I)$, there must
  be $x', y', z' \in I$ such that \begin{align*} x_{\neg i} &= x_{\neg
      i}'\\y_{\neg i} = z_{\neg i} &= y_{\neg i}' = z_{\neg i}'\\y_i'
    &\neq z_i'. \end{align*}
  Since $y_i' \neq z'_i$, we must either have that $x_i' \neq y_i'$ or
  $x_i' \neq z_i'$. In the former case, $\{x', y'\}$ is an edge of
  $K_t^n$ and in the latter case $\{x', z'\}$ is an edge of
  $K_t^n$. This contradicts the fact that $I$ is an independent set.
\end{proof}

Next we show that compressions can only decrease the size of the
vertex boundary.

\begin{claim}\label{claim:c-D} For all $i  \in [n]$ and $S \subseteq
  [t]^n$, $|\partial c_i(S)| \le |\partial S|$.
\end{claim}
\begin{proof} Fix $\bar{a} := a_1, \hdots, a_{i-1}, a_{i+1}, \hdots,
  a_n \in [t]$. Consider $T = \{(a_1, \hdots, a_{i-1})\}
  \times [t] \times \{(a_{i+1}, \hdots, a_n\} \subset [t]^n$.

  Note that for every vertex $v \in [t]^n$, $\partial \{v\} \cap
  T$ either has $0$ or $t-1$ elements. Thus, $|T
  \cap \partial S| \in \{0, t-1, t\}$. We claim that $|T \cap \partial
  c_i(S)| \le |T \cap \partial S|$ for all $T$.

  \begin{itemize}
  \item If $|T\cap \partial S| = 0$, then there are no edges between $S$ and
  $T$ and shifting the vertices of $S$ in the $i$th coordinate cannot
  change that. Thus, $|T \cap \partial c_i(S)| = 0$.

  \item If $|T\cap \partial S| = t-1$, then the set $\partial
  T \cap S$ must be constant in the $i$th coordinate. Thus,
  $c_i(\partial T \cap S) =
  \partial T \cap c_i(S)$ is constant in the $i$th coordinate, so $|T
  \cap \partial c_i(S)| = t-1$.

  \item If $|T\cap \partial S| = t$,
  then trivially $|T\cap \partial c_i(S)| \le t.$
  \end{itemize}

  Thus, summing $|T\cap \mathcal \partial c_i(S))| \le |T\cap \partial
  S|$ across all possible $T$, we have that $|\partial c_i(S)| \le
  |\partial S|$.
\end{proof}

\begin{rem}
The proof crucially uses the fact that $\partial S$ can include
elements of $S$. If we instead had defined the vertex boundary to be
$\partial S \setminus S$, there is a simple counterexample. Consider
$t = 3$ and $n = 2$ and $S = \{(1, 2), (1, 3), (2, 1), (3, 1)\}.$
Then it is not hard to check that $|\partial S| = |\partial c_1(S)| = 8$,
but $|\partial S \setminus S| = 4 < 5 = |\partial c_1(S) \setminus c_1(S)|$.
\end{rem}

\subsection{Proof of Theorem~\ref{thm:isoperimetry-tensor}}
Define
\begin{align}
\eta(t) &:= \frac{\log t}{\log t - \log (t-1)} = t\log t +
\Theta(\log t).
\end{align}
First, we show that $\eta(K_t^n) \le \eta(t)$. In fact, we show a
whole family of equality cases.
\begin{claim}\label{claim:tensor-poor-expansion}
  For all positive integers $n$ and $t$ such that $t \ge 3$,
  $\eta(K_t^n) \le \eta(t)$.
\end{claim}
\begin{proof}
  For all integers $k \in [n]$, consider $S = \{1\}^k \times
  [t]^{n-k}$. Then $\partial S = \{2, \hdots, t\}^k \times
  [t]^{n-k}$. Thus,
  \[
  \eta(K_t^n) \le \frac{\log \mu(S)}{\log \mu(\partial S)} =
  \frac{\log t^{-k}}{\log ((t-1)^kt^{-k})} = \frac{k\log
    \frac{1}{t}}{k \log \frac{t-1}{t}} = \eta(t).\qedhere
  \]
\end{proof}

The lower-bound is more difficult, we first need the following
inequality, proved in Appendix~\ref{app:ineq}.

  \begin{claim} \label{claim:bound-D-ineq} Let $t \ge 2$ be a positive
integer and let $x \ge y \ge 0$ be real numbers, then
    \begin{align} y^{1/\eta(t)} + (t-1)x^{1/\eta(t)} \ge (t-1)(x +
(t-1)y)^{1/\eta(t)} \label{eq:bound-D-ineq}
    \end{align}
  \end{claim}

\begin{lem}\label{lem:bound-D} For positive integers $n \ge 1$ and $t \ge
  3$ and all $S \subseteq [t]^n$, we have that 
\begin{align} \mu(\partial(S)) \ge
\mu(S)^{1/\eta(t)}. \label{eq:bound-D}
\end{align}
Therefore $\eta(K_t^n) \ge \eta(t)$.
\end{lem}

\begin{proof} By Claim~\ref{claim:c-D} and Remark~\ref{rem:compressed}, it suffices to consider the
  case that $S$ is compressed. We now proceed by induction on $n$.

  For our base case, $n = 1$, we must have that $S = \emptyset$ in
  which case (\ref{eq:bound-D}) is trivial, or $S = [k]$ for some
  positive integer $k \le t$. If $S = [1]$, then $\partial S = \{2,
  \hdots, t\}$, in which case we have an equality case of
  (\ref{eq:bound-D}) by the proof of Claim~\ref{claim:tensor-poor-expansion}. Otherwise, if $k \ge 2$, then $\partial S =
  [t]$, so $\mu(\partial S) = 1$, so (\ref{eq:bound-D}) holds.

  For $n \ge 2$, assume by the induction hypothesis that
  (\ref{eq:bound-D}) is true for all $S\subseteq \mathbb Z_t^m$ where $1
  \le m < n$. For all $i \in [t]$, let
  \begin{align}
  S_i &:= \{x_{\neg n} \mymid x_n \in S, x_n = i\} \label{eq:S_i}\\
  (\partial S)_i &:= \{x_{\neg n} \mymid x_n \in \partial S, x_n = i\}.
  \end{align}
  Since $S$ is compressed for all $1 \le i \le j \le t$, we have that
  $S_i \supseteq S_j$. Thus, if $i \in \{2, \hdots, t\}$ is nonzero,
  for any $x \in (\partial S)_i$, there is $y \in S_0$ connected to
  $x$ by an edge of $K_t^{n-1}$. Thus, $\partial S_0
  \subseteq (\partial S)_i$. Similarly, for any $x \in (\partial
  S)_0$, there is $y \in S_1$ such that $x$ is disjoint from
  $y$. Therefore, $\partial S_1 \subseteq (\partial S)_0$. Putting
  these together, 
  \begin{align*}
    \mu(\partial S) &= \frac{1}{t}\sum_{i\in[t]} \mu((\partial S)_i)\\
    &\ge \frac{1}{t}(\mu(\partial S_1) + (t-1)\mu(\partial S_0))\\
 &\ge \frac{1}{t}\left(\mu(S_1)^{1/\eta(t)} +
(t-1)\mu(S_0)^{1/\eta(t)}\right),
\end{align*} where we applied the inductive hypothesis in the last
step. Applying Claim \ref{claim:bound-D-ineq}, using the fact that $0
\le \mu(S_1) \le \mu(S_0)$, we have that
  \begin{align*}
    \mu (\partial(S)) &\ge \frac{1}{t}\left(\mu(S_1)^{1/\eta(t)} +
      (t-1)\mu(S_0)^{1/\eta(t)}\right)\\
    &\ge
    \frac{t-1}{t}\left(\mu(S_0) + (t-1)\mu(S_1)\right)^{1/\eta(t)}\\
    &\ge \frac{t-1}{t}\left(\sum_{i \in [t]}\mu(S_i)\right)^{1/\eta(t)}\\
    &= \left(\frac{1}{t}\sum_{i \in [t]} \mu(S_i)\right)^{1/\eta(t)}\\
    &= \mu(S)^{1/\eta(t)},
  \end{align*} as desired.
\end{proof}

Claim~\ref{claim:tensor-poor-expansion} and Lemma~\ref{lem:bound-D}
together imply Theorem~\ref{thm:isoperimetry-tensor}.

\subsection{A fine-tuned understanding of the isoperimetric
  profile.}

Recall the (vertex) isoperimetric profile of a graph $G$ to be
\[
\Phi(G, \nu) := \inf\{\mu(\partial S) \mymid \mu(S) \ge \nu\}.
\]
For $t \ge 3$ fixed, define
\[
\Phi_t(\nu) := \inf_{n \ge 1} \Phi(K_t^n, \nu).
\]
Note that $\Phi_t$ is non-decreasing. It is easier to work with
$\Phi_t(\nu)$ instead of each $\Phi(K_t^n, \nu)$ directly to avoid
complications with the discrete behavior of $\Phi(K_t^n, \nu)$ when
$n$ is small. By Theorem~\ref{thm:isoperimetry-tensor},
\begin{align}\Phi_t(\nu) \ge \nu^{1/\eta(t)}.\end{align} This is tight
whenever $\nu = t^{-k}$ for any integer $k\ge 0$, but ceases to be
tight when $\log_t(\nu)$ is non-integral (see Figure~\ref{fig:2}).

\begin{figure}
\begin{center}
\includegraphics[width=.6\textwidth]{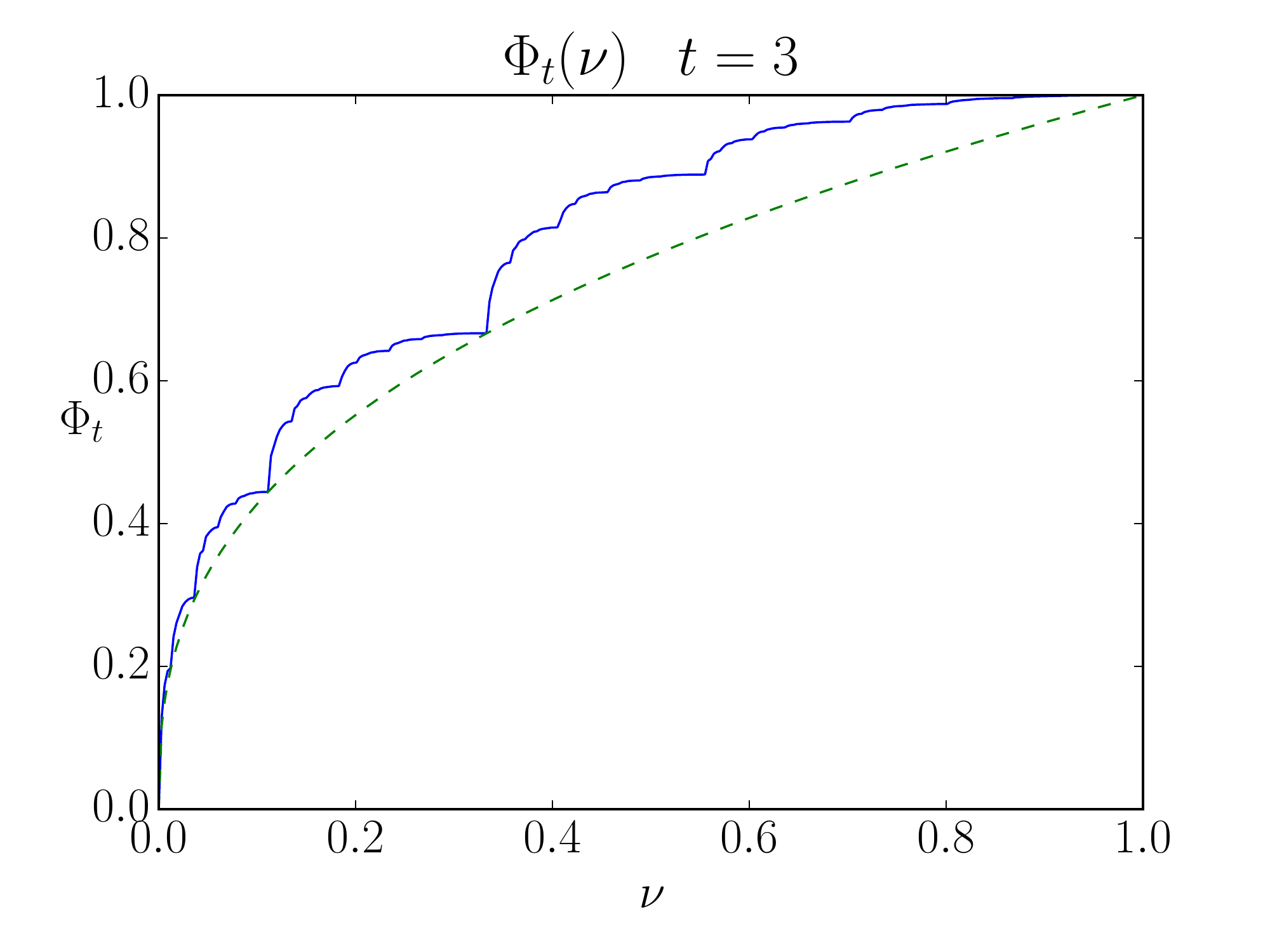}
\end{center}
\caption{A graph of $\Phi_t(\nu)$ for $t = 3$. The dashed curve
  $\nu^{1/\eta(t)}$ is for reference.}
\label{fig:2}
\end{figure}

The following recursive relationship allows one to compute
$\Phi_t(\nu)$ to arbitrary precision.

\begin{thm}\label{lem:bound-D-exact}
  For all $t \ge 3$,
  \begin{align} \Phi_t(\nu) = \begin{cases} \frac{t-1}{t}\Phi_t(t\nu) &
      \nu < 1/t\\ \frac{t-1}{t} + \frac{1}{t}\Phi_t\left(\frac{t\nu -
          1}{t-1}\right) & \nu \ge 1/t
    \end{cases}.\label{eq:bound-D-exact}
  \end{align}
\end{thm}
Using the simple fact that $\Phi_t(0) = 0$ and $\Phi_t(1) = 1$,
the above equation is extremely powerful. For example,
\[\Phi_3\left(\frac{5}{9}\right) = \frac{2}{3}
+ \frac{1}{3}\Phi_3\left(\frac{1}{3}\right) = \frac{8}{9},\] which is
an exact bound compared to $(\frac{5}{9})^{1/\eta(3)} \approx
\frac{7.24}{9}$. This recursion is what allowed the creation of
Figure~\ref{fig:2}.

Theorem~\ref{lem:bound-D-exact} is proved in
Appendix~\ref{app:profile}. This more refined understanding of $\Phi_t$ proves critical in the
combinatorial proof of Theorem~\ref{thm:tensor-stability-better}.

\section{Independent set stability results} \label{sec:stab}

\subsection{Black-box result for clique tensor powers} \label{subsec:stab-bb-clique}

First, we show that if a large independent set $I$ is somewhat close to a
maximum-sized independent set $J$, then it is really close to $J$. We
fix positive integers $n$ and $t \ge 3$.

\begin{lem}\label{lem:click}
  Let $I \subset [t]^n$ be an independent set with $\epsilon :=
  1 - t\mu(I).$ Assume there exists a maximum-sized independent set
  $J$ such that
  \[
  \mu(I \setminus J) < \frac{1}{t^3}.
  \]
  Then,
  \[
  \mu(I \setminus J) < 4\epsilon^{\eta(t)}.
  \]
\end{lem}

\begin{proof}
  Without loss of generality, we may assume that $J = [t]^{n-1} \times
  [1]$. Pick $J' = [t]^{n-1} \times \{j\}$ such that $j \neq 1$
  but otherwise $\mu(I \cap
  J')$ is maximal. Let $\delta := \mu(I \setminus J)$. Since $J$ and $J'$ are disjoint, we have that
  \[
  \mu(I \cap J') \ge \frac{\mu(I \setminus J)}{t -1} = \frac{\delta}{t - 1}.
  \]
  Now, consider $S = \partial (I \cap J')$. Recall the definition of
  $S_k \subseteq [t]^{n-1}$ from (\ref{eq:S_i}). Since $I \cap J'
  \subseteq J$ has the property that every element has the same last
  coordinate, $S_k = S_{k'}$ for all $k, k' \neq j$ and $S_j =
  \emptyset$. Thus, $\mu(S_k) = \frac{t}{t-1}\mu(S)$ for all $k \neq
  j$. Therefore,
  \[
  \mu(S \cap J) = \frac{1}{t}\mu((S\cap J)_i) = \frac{1}{t}\mu(S_i) = \frac{1}{t-1}\mu(S).
  \]
  Applying Theorem~\ref{thm:isoperimetry-tensor}, we get that
  \[
  \mu(S \cap J) = \frac{1}{t-1}\mu(\partial (I \cap J')) \ge
  \frac{1}{t-1}\mu(I \cap J')^{1/\eta(t)} \ge \frac{1}{t-1}\left(\frac{\delta}{t-1}\right)^{1/\eta(t)}.
  \]

  Since $I$ is an independent set, $\partial I$ is disjoint from
  $I$. Since $S \cap J = \partial (I \cap J') \cap J
  \subseteq \partial I$, we have that $I \cap J$ and $S \cap J$ are
  disjoint. Therefore,
  \begin{align}
    \mu(I \cap J) &\le \mu(J) - \mu(S \cap J)\le \frac{1}{t} -
    \frac{1}{t-1}\left(\frac{\delta}{t-1}\right)^{1/\eta(t)}. \label{eq:18}
  \end{align}
  But, we also know that
  \begin{align}
    \mu(I \cap J) = \mu(I) - \mu(I \setminus J)
    = \frac{1}{t}(1 - \epsilon) - \delta.\label{eq:19}
  \end{align}

  By (\ref{eq:18}) and (\ref{eq:19})
  \[
  \frac{1}{t}(1 - \epsilon) - \delta \le \frac{1}{t} -
    \frac{1}{t-1}\left(\frac{\delta}{t-1}\right)^{1/\eta(t)} =
    \frac{1}{t} - \frac{1}{t}\left(\frac{t\delta}{t-1}\right)^{1/\eta(t)}.
  \]

Thus,
\begin{align}
\epsilon \ge \left(\frac{t\delta}{t-1}\right)^{1/\eta(t)} -
 t\delta \ge \delta^{1/\eta(t)} - t\delta. \label{eq:20}
 \end{align}
\begin{figure}
\begin{center}
\includegraphics[width=.6\textwidth]{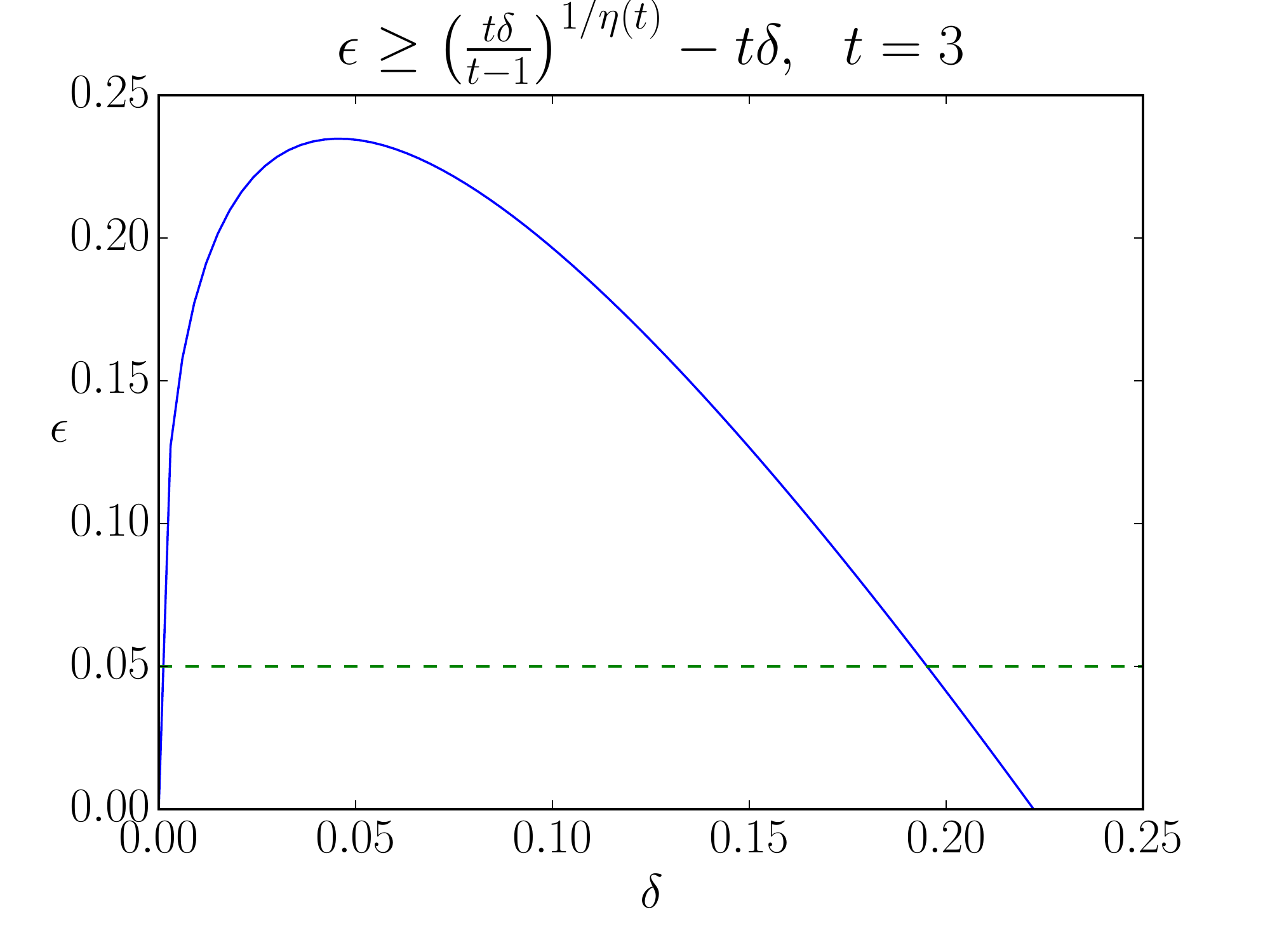}
\end{center}
\caption{Plot of (\ref{eq:20}) when $t = 3$. Notice the bifurcation of
solutions to (\ref{eq:20}) for a fixed $\epsilon$ (line $\epsilon = 0.05$ is dashed).}
\label{fig:1}
\end{figure}
Consider Figure~\ref{fig:1} which has a plot of the RHS of
(\ref{eq:20}) when $t = 3$. If $\epsilon$ is sufficiently small, then
the inequality holds only when $\delta$ is very small (polynomial in $\epsilon$) or very
large (about $\frac{1}{t}$). Since is `moderately' small ($\delta \le \frac{1}{t^3}$), we must have that $\delta$ is
very small. Quantitatively, note that
\begin{align*}
t\delta &= t\delta^{1/\eta(t)}\delta^{1-1/\eta(t)}\\
&\le t\delta^{1/\eta(t)}\left(\frac{1}{t^3}\right)^{1-1/\eta(t)}\\
&= t\delta^{1/\eta(t)}\frac{1}{t^3}\left(\frac{t^3}{(t-1)^3}\right)\\
&\le \frac{t\delta^{1/\eta(t)}}{(t-1)^3}.
\end{align*}
So
\[\epsilon \ge \delta^{1/\eta(t)}\left(1  - \frac{t}{(t-1)^3}\right).\]
Therefore,
\[\delta \le \left(\frac{(t-1)^3}{(t-1)^3 - t}\right)^{\eta(t)}\epsilon^{\eta(t)} \le
4\epsilon^{\eta(t)},\] where the last inequality follows from the
following claim which is proved in Appendix~\ref{app:ineq}.
\begin{claim}\label{claim:close-ind-ineq} For all $t \ge 3$,
 \[\left(\frac{(t-1)^3}{(t-1)^3 - t}\right)^{\eta(t)}\le 4.\]
\end{claim}
\end{proof}

We now use this lemma to `amplify' Theorem~\ref{thm:ADFS} to prove Theorem~\ref{thm:tensor-stability}.

\begin{proof}[Proof of Theorem \ref{thm:tensor-stability}.]
  Set $\epsilon_t := \frac{1}{C_t t^3} > 0$. Consider any independent
  set $I$ of of $K_t^n$ such that $\epsilon := 1 - t\mu(I) <
  \epsilon_t$. Pick any maximum-sized $J$ guaranteed by
  Theorem~\ref{thm:ADFS} such that
  \begin{align}
  \delta := \mu(I \setminus J) \le \mu(I \Delta J) \le C_t \epsilon < \frac{1}{t^3}.
  \end{align}
  By Lemma~\ref{lem:click}, we have that
  \[
  \delta \le 4\epsilon^{\eta(t)},
  \]
  as desired.
\end{proof}

\subsection{Improved stability result for clique tensor powers} \label{subsec:stab-combo}

In this section we improve $\epsilon_t$ in
Theorem~\ref{thm:tensor-stability} to an explicit expression. In fact,
we may show that
\[
\epsilon_t = 1 - \frac{3}{t} + \frac{2}{t^2}
\]
which corresponds to independent sets $I$ for which $\mu(I) > \frac{3t-2}{t^3}$.

First, we try to show that if an independent set $I$ is large enough,
then $I$ is either very close to or very far from a maximum-sized
independent set. To do this, we show that if $I$ is `moderately far'
from a maximum-sized independent set, then this moderate-sized portion
which is not in the maximum-sized independent set has such a large
vertex boundary that it precludes a large portion of the maximum-sized
independent set from being part of $I$, forcing the density of $I$ to
be at or below our threshold of $\frac{3t-2}{t^3}$.

We need a notation for the maximum sized
independent sets. For all $i \in [t]$ and $j \in [n]$ let
\begin{align}
J_{i,j} = [t]^{j-1} \times \{i\} \times [t]^{n - j}.
\end{align}

We say that $I$ is \emph{sorted} if there exists that for all $i_1,
i_2 \in [t]$ and $j \in [n]$ we have that $i_1 \le i_2$ implies that
\[
\mu(I \cap J_{i_1,j}) \le \mu(I \cap J_{i_2,j}).
\]

Note that unlike compressions, we may assume without loss of
generality that $I$ is sorted since permuting the labels so that an
independent set is sorted does not change its intersection sizes with
the maximum independent sets.

\begin{claim} \label{claim:ind-set-gap} Let $I \subset \mathbb [t]^n$
be a sorted independent set such that $\mu(I) > \frac{3t - 2}{t^3}$
(or $1 - t\mu(I) < \epsilon_t$), then
for all $j \in [n]$,
\begin{align}\mu(I \setminus J_{1,j}) < \frac{t-1}{t^4}\text{ or }\mu(I
\setminus J_{1,j}) > \frac{t-1}{t^3}.\label{eq:dicot}\end{align}
\end{claim}

\begin{proof} Without loss of generality, we may let $j = n$. Denote
  $J := J_{1,j}$. Let $\delta = \mu(I \setminus J)$. Since $I$ is an
  independent set
  \begin{align*}
    \mu (I \cap J) &\le \mu(J) - \mu(J \cap \partial(I \cap
    J_{2,n})).
  \end{align*}
  Note that $\mu(\partial (I \cap J_{2,n}) \cap J_{i, n})$ is $0$ if
  $i = 2$ but is $\frac{1}{t-1}\mu(\partial (I \cap J_{2,n}))$
  otherwise (see the proof of Theorem~\ref{thm:tensor-stability} for more
  explanation). Thus, by Theorem~\ref{thm:isoperimetry-tensor},

  \begin{align}
    \mu (I \cap J) &\le \mu(J) - \frac{1}{t-1}\mu(\partial (I \cap
    J_{2,n}))\label{eq:bound1}\\ 
    &\le \frac{1}{t} -
    \frac{1}{t-1}\Phi_t\left(\frac{\delta}{t-1}\right)\label{eq:bound2}\\
    &\le \frac{1}{t} - \frac{1}{t-1}\left(\frac{\delta}{t-1}\right)^{1/\eta(t)}.\label{eq:bound3}
  \end{align}

 Since $\mu(I) > \frac{3t - 2}{t^3}$, we have that
  \[ \frac{1}{t} + \delta -
\frac{1}{t-1}\left(\frac{\delta}{t-1}\right)^{1/\eta(t)} >
\frac{3t-2}{t^3}.
  \] Thus, we obtain that
  \begin{align} \frac{(t-2)(t-1)}{t^2} >
\left(\frac{t\delta}{t-1}\right)^{1/\eta(t)} - t\delta.\label{eq:RHS}
  \end{align} Note that the two sides of the inequality are equal at $\delta_j
= \frac{t-1}{t^4}$ and $\delta_j = \frac{t-1}{t^3}$. Note that since
$1/\eta(t) \in(0, 1)$ for all $t \ge 3$, the RHS of (\ref{eq:RHS}) is
concave for all $\delta \ge 0$. Thus, (\ref{eq:RHS}) is false when
$\delta \in [\frac{t-1}{t^4}, \frac{t-1}{t^3}]$. Therefore, we have (\ref{eq:dicot}).
\end{proof}

From Theorem~\ref{lem:bound-D-exact}, we can attain a bound that is even
better.

\begin{claim} \label{claim:ind-set-gap-better} Let $I \subset [t]^n$ be a sorted independent set such that $\mu(I) > \frac{3t - 2}{
t^3}$, then for all $j \in [n]$,
\begin{align}\mu(I \setminus J_{1,j}) < \frac{t-1}{t^4}\text{ or }\mu(I
\setminus J_{1,j}) > \frac{(2t-1)(t-1)}{t^4}.\label{eq:dicot2}\end{align}
\end{claim}

\begin{proof} Again, we may assume without loss of generality that $j
  = n$, let $J = J_{1,j}$. Let $\delta = \mu(I \setminus J)$. From
  Claim~\ref{claim:ind-set-gap}, we only need to consider the case
  that \begin{align}\frac{(2t-1)(t-1)}{t^4} \ge \delta > \frac{t-1}{t^3}.\label{eq:new_case}\end{align}

  From (\ref{eq:bound2})
  \[\mu (I \cap J) \le \frac{1}{t} -
\frac{1}{t-1}\Phi_t\left(\frac{\delta}{t-1}\right).
  \]
  Now make the substitution
  \[\delta = \frac{(t-1)}{t^3}(1 + \delta'),\] where $\delta' \in (0,
  \frac{t-1}{t}]$. From Theorem~\ref{lem:bound-D-exact},
  \begin{align*}
    \Phi_t\left(\frac{\delta}{t-1}\right) &= \Phi_t\left(\frac{1 +
        \delta'}{(t-1)^3}\right)\\
&= \frac{(t-1)^2}{t^2}\Phi_t\left(\frac{1 + \delta'}{t}\right)\\
&=
\frac{(t-1)^2}{t^2}\left(\frac{t-1}{t} +
\frac{1}{t}\Phi_t\left(\frac{\delta'}{t-1}\right)\right)\\
&\ge
\frac{(t-1)^2}{t^2}\left(\frac{t-1}{t} +
\frac{1}{t}\left(\frac{\delta'}{t-1}\right)^{1/\eta(t)}\right).
  \end{align*} Hence, since $\mu(I) > \frac{3t-2}{t^3}$,
  \[ \frac{t-1}{t^3}(1 + \delta') + \frac{1}{t} -
\frac{t-1}{t^2}\left(\frac{t-1}{t} +
\frac{1}{t}\left(\frac{\delta'}{t-1}\right)^{1/\eta(t)}\right)>
\frac{3t-2}{t^3}.
  \] Rearranging,
  \[ 0 > \left(\frac{\delta'}{t-1}\right)^{1/\eta(t)} - \delta'.
  \]
  Like in the proof of Claim~\ref{claim:ind-set-gap}, we have equality
  when $\delta' = 0$ and $\delta' = \frac{t-1}{t}$. Furthermore, since
  $1/\eta(t) \in (0, 1)$ for all $t \ge 3$, the RHS is concave when
  $\delta' \ge 0$. Thus, the inequality is false for all $\delta \in
  (0, \frac{t-1}{t}].$ Therefore, (\ref{eq:new_case}) can never hold,
  proving (\ref{eq:dicot2}), as desired.
\end{proof}

The next key step is to show Theorem~\ref{thm:tensor-stability-better}
essentially holds for \emph{compressed} independent sets $I$.

\begin{lem}\label{lem:ind-set-collapse} Let $I \subset \mathbb [t]^n$
be a compressed independent set such that $\mu(I) > \frac{3t - 2}{t^3}$,
then for some $j \in [n]$, 
\begin{align}
\mu(I \setminus J_{1,j}) < \frac{t-1}{t^4}.\label{eq:compressed_claim}
\end{align}
\end{lem}

Note that by Lemma~\ref{lem:click}, we immediately have that
Theorem~\ref{thm:tensor-stability-better} holds for compressed
independent sets.

\begin{proof} We prove this statement by induction on $n$. If $n = 1$,
  then the bound holds since $I = \{(1)\}$ which is clearly a
  maximum-sized independent set. Now assume $n \ge 2$ and that the
  (\ref{eq:compressed_claim}) holds for all compressed independent
  sets $I \subset \mathbb [t]^{n-1}$ with $\mu(I) >
  \frac{t-1}{t^3}$.

  Fix a compressed independent set $I \subseteq [t]^n$ with $\mu(I)
  \ge \frac{t-1}{t^3}$. From Claim~\ref{claim:ind-set-gap-better},
  if the lemma is false, then we have that for all $j \in [n]$,
  \[\mu(I \setminus J_{1,j}) > \frac{(2t-1)(t-1)}{t^4}.\]
  Since $I$ is compressed, this implies that for all such $j$
  \[\mu(I \cap J_{2,j}) > \frac{2t-1}{t^4}.\]

  Recall that for all $a \in [t]$, $I_a = \{(x_1, \hdots, x_{n-1})
  \mymid (x_1, \hdots, x_{n-1}, a) \in I\} \subseteq [t]^{n-1}.$ We claim that $I_2$ is an
  independent set of $K_t^{n-1}$. Note that in general $I_1$ is
  \emph{not} an independent set of $K_t^{n-1}$. Since $I$ is
  compressed, $I_2 \subseteq I_1$. Thus, if there were $x, y \in
  I_2$ which form an edge of $K_t^{n-1}$, then $(x, 1), (y, 2) \in I$
  form an edge of $K_t^n$, contradicting that $I$ is an independent
  set. Therefore, $I_2 \subseteq [t]^{n-1}$ is indeed an independent
  set.

  Note that $\mu(I_2) = t\mu(I \cap J_{2,n}) > \frac{2t-1}{t^3}$ which
  is not sharp enough of a lower bound to invoke the inductive
  hypothesis. But, we claim that we can find a compressed independent set $\tilde{I}
  \subseteq I_1$ such that $\mu(\tilde{I}) \ge \mu(I) > \frac{3t-2}{t^3}$.

  Pick $a \in [t]$ such that $(I_1 \setminus I_2) \cap J_{a,
    n-1}\subseteq [t]^{n-1}$ has maximal size.\footnote{To keep
    notation as concise as possible, we use the $J_{i,j}$ notation to
    refer to both the maximal independent sets of $[t]^{n-1}$ and
    $[t]^n$. It should be clear from context which we are referring
    to.} Note that since $I_1 \setminus I_2$ is not necessarily
  compressed, $a$ might not equal $1$.  Let $\hat{I} = I_2 \cup ((I_1
  \setminus I_2) \cap J_{a, n-1})$. We claim that $I$ is an
  independent set (although it might not be compressed). As previously
  established $I_2$ is an independent set and clearly $(I_1 \setminus
  I_2) \cap J_{a, n-1}$ is an independent set since the last
  coordinate is constant. Thus, if $I$ were not an independent set
  then, there is $x \in I_2$ and $y \in I_1 \setminus I_2$ which are
  connected by an edge in $K_t^{n-1}$. But, note that $(x, 2), (y, 1)
  \in I$ are connected by an edge in $K_t^n$, contradiction. Thus,
  $\hat{I}$ is an independent set of $K^{n-1}_t$.

  Let $\tilde{I}$ be a compression of $\hat{I}$. since $I_2$ and $I_1$ are
  already compressed and $I_2 \subseteq \hat{I} \subseteq I_1$, we
  have that $I_2 \subseteq \tilde{I} \subseteq I_1$. Now,
  \begin{align*}
  \mu(\tilde{I}) &= \mu(\hat{I})\\
  &\ge \mu(I_2) + \frac{\mu(I_1) -
    \mu(I_2)}{t}\\
   &= \frac{\mu(I_1) + (t-1)\mu(I_2)}{t}\\
   &\ge \frac{1}{t}\sum_{i=1}^t \mu(I_i)\\
   &= \mu(I) > \frac{3t-2}{t^3}.
  \end{align*}

  Thus, we may now invoke the induction hypothesis on $\tilde{I}$. Therefore,
  there exists $j \in [n-1]$ such that
  \[
  \mu(\tilde{I} \setminus J_{1,j}) < \frac{t-1}{t^4}.
  \]
  Since $I_2 \subseteq \tilde{I}$, we have that 
  \[
  \mu(I_2 \setminus J_{1,j}) \le \mu(\tilde{I} \setminus J_{1,j}) < \frac{t-1}{t^4}.
  \]
  Therefore, since $I$ is compressed
  \begin{align}
  \mu(I \setminus (J_{1,j} \cup J_{1,n})) &= \frac{1}{t}\sum_{i=2}^n
  \mu(I_i \setminus J_{1,j})\\
  &\le \frac{t-1}{t}\mu(I_2 \setminus J_{1,j})\\
  &\le \frac{(t-1)^2}{t^5}.
  \end{align}
  Hence, recalling that $I$ is very far from $J_{1,n}$
  \begin{align}
    \mu((I \setminus J_{1,n}) \cap J_{1,j}) &= \mu(I \setminus J_{1,n}) -
    \mu(I \setminus (J_{1,j} \cup J_{1,n}))\\
    &\ge \frac{(2t-1)(t-1)}{t^4} - \frac{(t-1)^2}{t^5} = \frac{(2t^2 - 2t + 1)(t-1)}{t^5}.
  \end{align}
  Likewise,
  \begin{align}
    \mu((I \setminus J_{1,j}) \cap J_{1,n}) &= \mu(I \setminus J_{1,j}) -
    \mu(I \setminus (J_{1,j} \cup J_{1,n}))\\
    &\ge \frac{(2t-1)(t-1)}{t^4} - \frac{(t-1)^2}{t^5} = \frac{(2t^2 - 2t + 1)(t-1)}{t^5}.
  \end{align}

Let $I' = I \cap J_{2,j} \cap J_{1,n}$ and $I'' = I \cap J_{1,j} \cap
J_{2,n}$. Now observe that since $I$ is compressed
\begin{align*}
\mu(I') = \mu(I \cap J_{2,j}\cap J_{1,n}) &\ge \frac{1}{t-1}\mu((I \setminus J_{1,j})
\cap J_{1,n})= \frac{2t^2 - 2t + 1}{t^5}.
\end{align*}
Similarly,
\begin{align*}
\mu(I'') = \mu(I \cap J_{1,j}\cap J_{2,n}) &\ge \frac{1}{t-1}\mu((I \setminus J_{1,n})
\cap J_{1,j})= \frac{2t^2 - 2t + 1}{t^5}.
\end{align*}

Since $I'$ is constant in both the $j$th and $n$th coordinates,
\[
\mu(\partial I' \cap J_{1,j} \cap J_{2,n}) =
\frac{1}{(t-1)^2}\mu(\partial I') \ge \frac{1}{(t-1)^2}\Phi_t(\mu(I')).
\]

From Theorem~\ref{lem:bound-D-exact}, we have that
  \begin{align*}\Phi_t(\mu(I')) &\ge \Phi_t\left(\frac{1}{t^3} +
\frac{(t-1)^2}{t^5}\right)\\
&= \frac{(t-1)^2}{t^2}\left(\frac{t-1}{t} +
\frac{1}{t}\Phi_t\left(\frac{t-1}{t^2}\right)\right)\\
&\ge \frac{(t-1)^3}{t^3}
  \end{align*}
  since $\Phi_t(\nu) \ge 0$.  Therefore, since $I' \cup
  I''$ is an independent set
  \begin{align*} \frac{1}{t^2} &= \mu(J_{1,j} \cap J_{2,n})\\
    &\ge \mu(I'') + \mu(\partial I'\cap
J_{1,j} \cap J_{2,n})\\ &\ge \frac{2t^2 - 2t + 1}{t^5} +
\frac{1}{(t-1)^2}\Phi_t(\mu(I'))\\ &\ge \frac{2t^2 - 2t + 1}{t^5} +
\frac{t-1}{t^3}\\&= \frac{t^3 + t^2 - 2t + 1}{t^5} >
\frac{1}{t^2},\text{ (since $t \ge 3$)}
  \end{align*} contradiction. Thus, the lemma is true. 
\end{proof}

Now we extend this result to sorted independent sets; and thus all
independent sets. 

\begin{lem}\label{lem:ind-set-collapse-full} Let $I \subset [t]^n$ be
  a sorted independent set such that $\mu(I) > \frac{3t - 2}{t^3}$,
  then for some $j \in [n]$,
  \begin{align}
  \mu(I \setminus J_{1,j}) <
  \frac{t-1}{t^4}.\label{eq:win}
  \end{align}
\end{lem}
\begin{proof} Like in the proof of Lemma~\ref{lem:ind-set-collapse},
  by Claim~\ref{claim:ind-set-gap-better}, we may assume for sake of
  contradiction that for all $j \in [n]$,
  \[\mu(I \setminus J_{1,j}) >
  \frac{(2t-1)(t-1)}{t^4}.\]
  It is not hard to see that for all $i, j \in [n]$ such that $i \neq
  j$,
  \begin{align}
    \mu(c_i(I) \setminus J_{1,j}) = \mu(I \setminus J_{1,j}) > \frac{(2t-1)(t-1)}{t^4}.\label{eq:comp_triv}
  \end{align}
  We seek to show that for all $j \in [n]$,
  \begin{align}
  \mu(c_j(I) \setminus J_{1,j}) >
  \frac{(2t-1)(t-1)}{t^4}. \label{eq:comp_nontriv}
  \end{align}

  By Claim~\ref{claim:ind-set-gap-better}, assume for sake of
  contradiction that \begin{align}\mu(c_j(I) \setminus J_{1,j}) < \frac{t-1}{t^4}\label{eq:325}\end{align} for
  some $j \in [n]$. We may assume without loss of generality that $j =
  n$.  Since $I$ is sorted, 
  \begin{align}\mu(I \cap J_{2,n}) \ge \frac{1}{t-1}\mu(I \setminus J_{1,n}) > 
  \frac{2t-1}{t^4}.\label{eq:7}\end{align}

  Therefore,
  \[
  \mu(\partial (I \cap J_{2,n})) \ge \Phi_t\left(\frac{2t-1}{t^4}\right) =
  \frac{(t+1)(t-1)^3}{t^4}.
  \]
  This implies that
  \[
  \mu(\partial(I \cap J_{2,n}) \cap J_{1,n}) = \frac{1}{t-1}\mu(\partial(I \cap J_{2,n}))
  = \frac{(t+1)(t-1)^2}{t^4}.
  \]
  Observe that since $I$ is an independent set \[\mu(\partial(I \cap
  J_{2,n}) \cap I) = 0.\] Therefore, if $x \in \partial(I \cap
  J_{2,n}) \cap c_n(I)$, then $(x_1, \hdots, x_{n-1}, 1) \in I$
  (because any other choice for the last coordinate would violate the
  above relation). Therefore, \begin{align}\mu(\partial(I \cap J_{2,n}) \cap
  J_{1,n} \cap c_n(I)) \le \mu(I \cap J_{2,n}).\label{eq:5}\end{align} 
  From this, we get that
  \begin{align}
  \mu(J_{1,n} \setminus c_n(I)) &\ge \mu((\partial(I \cap J_{2,n}) \cap
  J_{1,n}) \setminus c_n(I))\\&= \mu((\partial(I \cap J_{2,n}) \cap
  J_{1,n}) ) - \mu((\partial(I \cap J_{2,n}) \cap
  J_{1,n}) \cap c_n(I))\\
  &\ge \mu((\partial(I \cap J_{2,n}) \cap
  J_{1,n}) ) - \mu(I \cap J_{2,n})\text{ (by (\ref{eq:5}))}\label{eq:678}
  \end{align}

  Next, we deduce
  \begin{align} \mu (I) &= \mu(c_n(I) \cap J_{1,n}) + \mu(c_n(I)
\setminus J_{1,n})\\ &< \frac{1}{t} - \mu(J_{1,n}\setminus c_n(I)) +
\frac{t-1}{t^4}\text{ (by (\ref{eq:325}))}\\ &\le \frac{1}{t} - (\mu(\partial(I\cap J_{2,n}) \cap
J_{1,n}) - \mu(I \cap J_{2,n})) + \frac{t-1}{t^4}\text{ (by (\ref{eq:678}))}\label{eq:99}
  \end{align}

  Let $\nu := \mu(I \cap J_{2,n})$. Then note that
\[\mu(\partial(I \cap J_{2,n}) \cap J_{1,n}) = \frac{1}{t-1}\mu(\partial(I
\cap J_{2,n})) \ge \frac{1}{t-1}\Phi_t(\nu).\] Thus, by (\ref{eq:99})
  \begin{align} \mu(I) < \frac{t^3 + t - 1}{t^4} -
\left(\frac{1}{t-1}\Phi_t(\nu) - \nu\right).\label{eq:67}
  \end{align}

  We divide the remainder of the proof into three
cases depending on the value of $\nu$.

  Case\footnote{Recall that $\nu > \frac{2t-1}{t^4}$ by
    (\ref{eq:7})} 1: $\frac{2t-1}{t^4} < \nu \le \frac{1}{t^2}$.   By
  Theorem~\ref{lem:bound-D-exact} and the fact that $\Phi_t(\rho) \ge
  \rho$ for all $\rho \in [0, 1]$,
  \begin{align*} \Phi_t(\nu) &= \frac{(t-1)^2}{t^2}\left(\frac{t-1}{t} +
\frac{1}{t}\Phi_t\left(\frac{t^3\nu - 1}{t-1}\right)\right)\\ &=
\frac{(t-1)^2}{t^2}\left(\frac{t-1}{t} +
\frac{1}{t}\left(\frac{t-1}{t} + \frac{1}{t}\Phi_t\left(\frac{t^4\nu -
(2t-1)}{(t-1)^2}\right)\right)\right)\\ &\ge \frac{(t-1)^3(t+1) +
t^4\nu - (2t-1)}{t^4}.
  \end{align*} Thus, by (\ref{eq:67})
  \[ \frac{3t - 2}{t^3} < \mu(I) < \frac{t^3 + t - 1}{t^4} -
\frac{1}{t-1}\cdot \frac{(t-1)^3(t+1) + t^4\nu - (2t-1)}{t^4} + \nu.
  \] Rearranging,
  \[
  \frac{(2t-1) + 2(t-1)^3}{(t-1)t^4} \le
  \left(\frac{t-1}{t-2}\right)\nu \le \frac{t-2}{t^2(t-1)}.
  \]
  This implies that
  \[ 2(t-1)^3 + 2t-1 < t^3 - 2t^2.
  \] Thus, $t^3 - 4t^2 + 8t - 2 < 0$, but this is false for $t \ge 3$,
contradiction.

  Case 2, $\frac{1}{t^2} < \nu \le \frac{(2t-1)(t-1)}{t^4}$.

  Then $\Phi_t(\nu) \ge
\frac{(t-1)^2}{t^2}$. Thus, by (\ref{eq:67})
  \[ \mu(I) < \frac{t^3 + t - 1}{t^4} - \frac{t-1}{t^2} +
\frac{(2t-1)(t-1)}{t^4} = \frac{3t^2 - 2t}{t^4} = \frac{3t -2}{t^4} < \mu(I),
  \] contradiction.
  
  Case 3, $\nu > \frac{(2t-1)(t-1)}{t^4}$.

  Observe that \begin{align*} \Phi_t(\nu) \ge
\Phi_t\left(\frac{2t^2 - 3t + 1}{t^4}\right) &=
\frac{t-1}{t}\Phi_t\left(\frac{2t^2 - 3t + 1}{t^3}\right)\\&= \frac{(t-1)^2}{t^2} +
\frac{t-1}{t^2}\Phi_t\left(\frac{t^2-3t+1}{t(t-1)}\right)\\ &\ge
\frac{t(t-1)^2 + (t^2-3t+1)}{t^3}\\ &> \frac{t^2(t-1)^2 +
(t-1)(t^2-3t+1)}{t^4}\text{ (since $t \ge 3$)}\\
&= \frac{(t-1)(t^3 - 3t + 1)}{t^4}.
  \end{align*}
  Since $I$ is sorted, $\mu(I) \ge 2\nu$. Therefore,
  \[ 2\nu \le \mu(I) < \frac{t^3 + t -1}{t^4} - \frac{t^3 - 3t +
1}{t^4} + \nu.
  \] Thus, $\nu < \frac{4t - 2}{t^4}$, but $\frac{4t-2}{t^4} \le \frac{(2t-1)(t-1)}{t^4}$
for $t \ge 3$, contradiction.

 End Cases.

  Therefore, our assumption that~(\ref{eq:comp_nontriv}) failed to
  hold is false. Therefore
  \[\mu(c_j(I) \setminus J_{1,i}) > \frac{(2t-1)(t-1)}{t^4}.\] for all $i,j
  \in [n]$. Applying this fact repeatedly, we can find a compressed
  $I'$ of the same cardinality as $I$ such that $\mu(I'
  \setminus J_{1,i}) > \frac{(2t-1)(t-1)}{t^4}$ for all $i \in [n]$,
  contradicting Lemma \ref{lem:ind-set-collapse}. Thus, our
  counterexample $I$ could have
  never existed. This proves the Lemma.
\end{proof}

\begin{proof}[Proof of Theorem~\ref{thm:tensor-stability-better}]
Let $I \subset [t]^n$ be an independent set with $\mu(I) > \frac{3t-2}{t^3}$. Assume without loss of
generality that $I$ is sorted. By
Lemma~\ref{lem:ind-set-collapse-full}, we know that there is $j \in
[n]$ such that
\[
\mu(I \setminus J_{1,j}) \le \frac{t-1}{t^4} < \frac{1}{t^3}.
\]
Thus, by Lemma~\ref{lem:click}, we have that
\[
\mu(I \setminus J_{1,j}) \le 4\epsilon^{\eta(t)},
\]
as desired.
\end{proof}

\section*{Acknowledgments}\label{sec:ack}

The author is indebted to Venkatesan Guruswami for numerous insightful
discussions and comments, in particular for pointing the author to
\cite{Alon}.

The author would also like to thank Boris Bukh and Po-Shen Loh for
helpful comments and discussions.

The 2D plots were created using Matplotlib~\cite{Hunter2007}. The 3D
visualizations were created using Asymptote~\cite{hammerlindl2014asymptote}.

\bibliographystyle{alpha} \bibliography{../bib/jabref}

\appendix

\section{Proofs of algebraic inequalities}\label{app:ineq}

\begin{proof}[\textbf{Proof of Claim~\ref{claim:bound-D-ineq}}] For $c
  \ge 0$, let $f_c(z) = (z+c)^{\alpha(t)} - z^{\alpha(t)}$. Notice
  that if $z >0$, then $f'_c(z) =
  (\alpha(t))((z+c)^{\alpha(t)-1}-z^{\alpha(t)-1}) \le 0$. Thus, we
  have that $(t-1)f_c(y) \ge (t-1)f_c(x)$ for all $c \ge 0$. Consider
  $c =(t-1)y$; we then have that
\begin{align*} (t-1)f_c(y) &= (t-1)((ty)^{\alpha(t)} -
y^{\alpha(t)}) = (t-1)(t^{\alpha(t)} - 1)y^{\alpha(t)} = y^{\alpha(t)}
\ge\\ (t-1)f_c(x) &= (t-1)((x + (t-1)y)^{\alpha(t)} - x^{\alpha(t)}).
\end{align*} Rearranging, we obtain (\ref{eq:bound-D-ineq}).
\end{proof}

\begin{proof}[\textbf{Proof of Claim~\ref{claim:close-ind-ineq}.}]
  First, verify the cases $t = 3$ and $t = 4$ using a calculator. Notice that
  $\eta(t) = \frac{\log t}{\log t - \log(t-1)}\le t\log t$ so
  \[
  \left(\frac{(t-1)^3}{(t-1)^3 - t}\right)^{\eta(t)} \le
  e^{\frac{t\eta(t)}{(t-1)^3 - t}} \le e^{\frac{t^2\log t}{(t-1)^3 - t}}.
  \]
  Also use a calculator to verify that $h(t) := \frac{t^2\log
    t}{(t-1)^3 - t}$ is less than $1$ for $t = 5$. Now observe that
  when going from $t$ to $t+1$, the numerator increases by
  \begin{align*}
  (t+1)^2\log(t+1) - t^2\log t &= (2t+1)\log(t+1) + t^2\log(1 +
  \frac{1}{t})\\
  &\le (2t+1)\log(t+1) + t \le (2t+1)t + t\\
  &= 2t^2 + 2t.
  \end{align*}
  and the denominator increases by
  \[
  t^3 - (t+1) - (t-1)^3 + t = 3t^2 - 3t
  \]
  Since $2t^2 + 2t \le 3t^2 - 3t$ for all $t \ge 5$ and $h(5) \le 1,$
  we have by a simple inductive proof that $h(t) \le 1$ for all $t \ge
  5$. Thus, for all $t \ge 5$,
  \[
  \left(\frac{(t-1)^3}{(t-1)^3 - t}\right)^{\eta(t)} \le e^{1} < 4,
  \]
  as desired.
\end{proof}

\section{Proof of Theorem~\ref{lem:bound-D-exact}}\label{app:profile}

The first step in proving this theorem is to determine the structure
of $S$ when $\mu(S)$ is fixed but $\mu(\partial S)$ is minimized. In
particular, we need $S$ to look as much like a maximal independent set
(e.g., $J = [t]^{n-1} \times [1]$) as possible.

\begin{claim}\label{claim:pack1} Let $t \ge 3$ and $n$ be positive
  integers. Let $J$ be a maximum-sized independent set. Consider
  $S\subseteq \mathbb [t]^n$.
  \begin{enumerate}
  \item If $\mu(S) < \frac{1}{t}$, then there exists $S' \subset
    [t]^n$ such that $\mu(S') = \mu(S)$, $\mu(\partial S') \le
    \mu(\partial S)$, and $S' \subset J$.

  \item If $\mu(S) \ge \frac{1}{t}$, then there exists $S' \subset [t]^n$ such
    that $\mu(S') = \mu(S)$, $\mu(\partial S') \le
    \mu(\partial S)$, and $J \subseteq S'$.
  \end{enumerate}
\end{claim}

 For each $x \in S$, define $|x|$, the \textit{level} of $x$,
be the number of coordinates of $x$ not equal to $1$ (c.f.,
\cite{Alon}).

\begin{proof} Without loss of generality, assume that $J = [t]^{n-1}
  \times [1]$. By Claim \ref{claim:c-D}, we may assume that $S$ is
compressed. This immediately resolves the case $n = 1$, so we may
assume $n \ge 2$.

Consider the map $\Pi : \mathbb [t]^n \to \{0, 1\}^n$
such that
\begin{align}
\Pi(x)_i := \begin{cases}0 & x_i = 1\\ 1 & x_i \neq 1\end{cases}\text{
  for all $i \in [n]$.}
\end{align}
Let $f_S := \mathbf 1_{\Pi(S)} : \{0,1\}^n \to \{0, 1\}$
be the indicator function of $\Pi(S)$. Since $S$ is compressed, $f_S$ is a
\emph{monotone} Boolean function: $f_S(x) \le f_S(y)$ whenever $x_i \le
y_i$ for all $i \in [n]$.

For all $z \in \{0, 1\}^n$, let $\neg z$ denote the bitwise complement
of $z$. Note that for any $x \in \Pi^{-1}(z)$ and $y \in \Pi^{-1}(z)$,
$x$ and $y$ are connected by an edge in $K_t^n$. Therefore, because $S$ is
compressed
\begin{align}
\partial S = \bigcup_{z \in \Pi(S)}\Pi^{-1}(\neg z)
\end{align}
and so
\begin{align} \mu(\partial S) = \frac{1}{t^n}\sum_{z \in \Pi(S)} |\Pi^{-1}(\neg z)| = \frac{1}{t^n}\sum_{z \in
\Pi(S)}(t-1)^{n-|z|}. \label{eq:mu_mon}
\end{align}

  We now describe an algorithm which modifies $S$ into a compressed $S'$
such that $\mu(S' \cap J)$ is maximized while keeping $\mu(\partial S') \le
\mu(\partial S)$ and $\mu(S) \le \mu(S')$. This algorithm consists of
two subroutines.

    \begin{figure}
      \begin{center}
        \includegraphics[width = 2.5in]{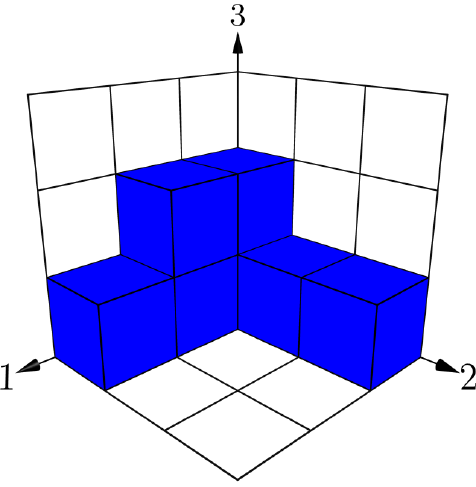}\hspace{1in}
        \includegraphics[width = 2.5in]{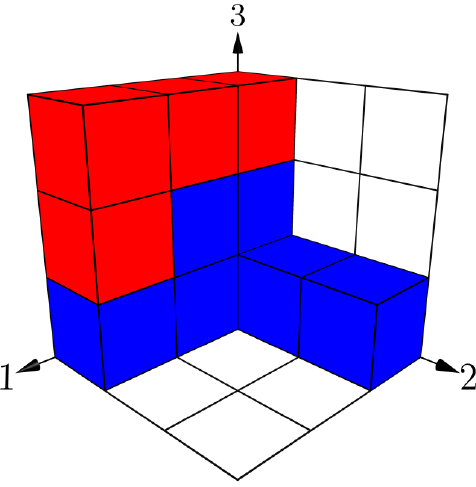}
      \end{center}
      \caption{A visualization of the operation $\Fill(S)$ when $n = t =
        3$. Each cube represents an element of $S$, with the red cubes
        being the ones that are changed. Each axis label represents a
        coordinate. For example, the red cube in the upper-left-hand corner
        represents the vertex $(3, 1, 3)$ of $K_3^3$.}
      \label{fig:3_fill}
    \end{figure}

    \textit{Filling.} See Figure~\ref{fig:3_fill}.
    Let 
    \[
    \Fill(S) = \bigcup_{z \in \Pi(S)} \Pi^{-1}(z).
    \]
    Note that $S \subseteq \Fill(S)$ but $\Pi(S) = \Pi(\Fill(S))$, so
    $\mu(\partial(\Fill(S))) = \mu(\partial S)$ by (\ref{eq:mu_mon}).

    Note that $\Fill(S)$ is compressed since $\mathbf 1_{\Pi(S)}$ is monotone.

    \begin{figure}
      \begin{center}
        \includegraphics[width = 2.5in]{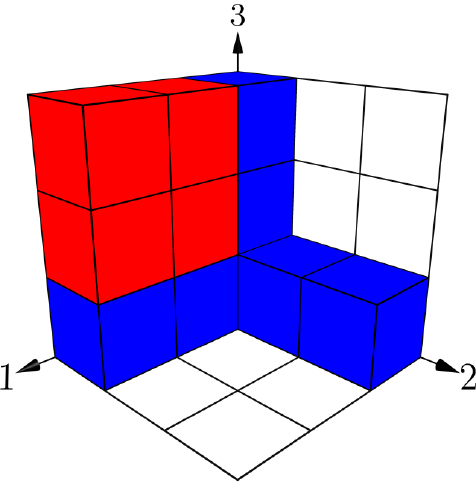}\hspace{1in}
        \includegraphics[width = 2.5in]{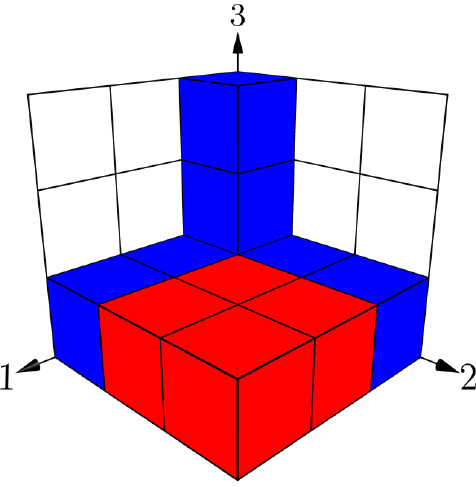}
      \end{center}
      \caption{A visualization of the operation $\Fold_{\{2\}}(S)$ when $n =
        t = 3$. See the caption for Figure~\ref{fig:3_fill} on
        interpreting this visualization.}
      \label{fig:3_fold}
    \end{figure}

        \begin{figure}
      \begin{center}
        \includegraphics[width = 2.5in]{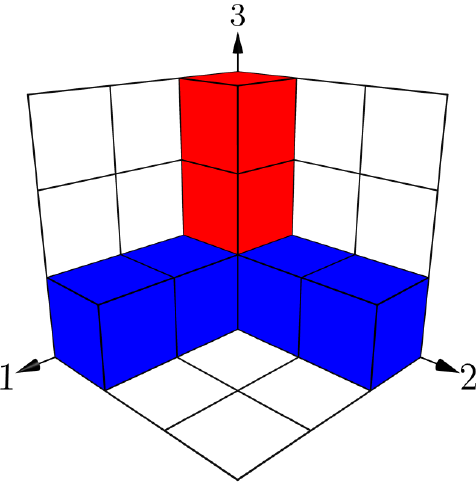}\hspace{1in}
        \includegraphics[width = 2.5in]{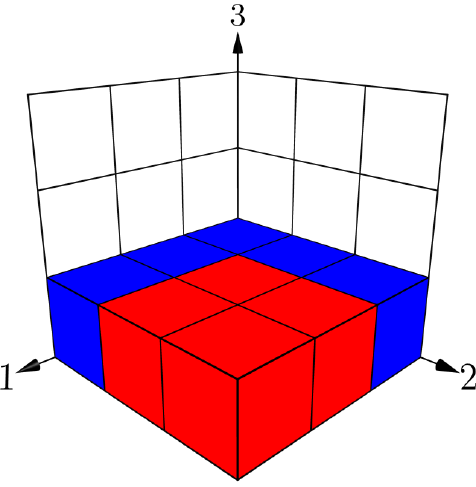}
      \end{center}
      \caption{A visualization of the operation $\Fold_{\{1,2\}}(S)$ when $n = t
        = 3$. Note that $\mu(S) < \mu(\Fold_{\{1,2\}}(S))$. See the caption for Figure~\ref{fig:3_fill} on
        interpreting this visualization.}
      \label{fig:3_flat}
    \end{figure}

    \textit{Folding.}\footnote{Note that this Folding operation is
      considered another form of compression in the literature,
      although typically used for Kneser graphs. For example see
      \url{https://gilkalai.wordpress.com/2008/10/06/extremal-combinatorics-iv-shifting/}}
    Assume that $S = \Fill(S).$ That is, for each $z \in \Pi(S)$,
    $\Pi^{-1}(z) \subseteq S.$

    The operator $\Fold_A$ is defined for each subset $A \subseteq
    [n-1]$. 

    For each $B \subseteq [n]$ let $\sigma_B : \mathbb \{0, 1\}^n \to
    \mathbb \{0, 1\}^n$ be the operator which negates the elements
    indexed by $B$
    \[
    \sigma_B(x)_i = \begin{cases}\neg x_i & i \in B\\ x_i & i \not\in B\end{cases}.
    \]
    For any $A \subseteq [n-1]$ let
    \begin{align}
    F_A = \{x \in \Pi(S) : x_A = 0, x_n = 1, \sigma_{A \cup \{n\}}(x)
    \in \Pi(S)\}.
    \end{align}

    Then, we define
    \[
    \Fold_A(S) := \Pi^{-1}[(\Pi(S) \setminus F_{A}) \cup \sigma_{A \cup \{n\}}(F_{A})].
    \]
    Figures~\ref{fig:3_fold} and \ref{fig:3_flat} help to visualize
    this operator.

    First, note that in the case $A = \emptyset$, $F_{\emptyset} = 0$
    since $S$ is compressed. Thus, since $S = \Fill(S) = \Pi^{-1}(\Pi(S))$, $\Fold_A(S) = S$.

    For $A \neq \emptyset$, note that since each element of $x \in \Pi(S)$
    either stays the same or is replace by $y \in \Pi(\Fold_A)$ such
    that $|x| \le |y|$. Thus, since $|\Pi^{-1}(y)| \ge |\Pi^{-1}(x)|$
    for all such $x$ and $y$. we have that $\mu(\Fold_i(S))
    \ge \mu(S)$. Furthermore, by (\ref{eq:mu_mon}), if we know that $\Fold_A(S)$ is
    compressed, then $\mu(\partial \Fold_A(S)) \le \mu(\partial S)$.

    Thus, it suffices to determine when $\Fold_A(S)$ is compressed. We
    claim that this is always the case when $\Fold_B(S) = S$ for all
    $B \subsetneq A$.

    \begin{claim}\label{claim:fold}
      Let $S \subseteq [t]^n$ be compressed and $A \subseteq [n-1]$
      nonempty. If $S = \Fill(S)$ and $\Fold_B(S) = S$ for all $B
      \subsetneq A$, then $\Fold_A(S)$ is compressed and so by the
      above discussion $\mu(\Fold_i(S)) \ge \mu(S)$ and $\mu(\partial
      \Fold_A(S)) \le \mu(\partial S)$.
    \end{claim}
    \begin{proof}
This is equivalent to showing that $\mathbf
    1_{\Pi(\Fold_A(S))} = \mathbf 1_{(\Pi(S) \setminus F_{A}) \cup
      \sigma_{A \cup \{n\}}(F_{A})}$ is monotone. Assume for contradiction that there is $x \in
    \Pi(\Fold_A(S))$ and $y \in \{0, 1\}^n \setminus
    \Pi(\Fold_A(S))$. such that $y \le x$.

    First consider the case $x \in \sigma_{A\cup\{n\}}(F_{A})$.Thus, $x_i
    = 1$ for all $i \in A$ and $x_n = 0$. Since $y \le x$, $y_n =
    0$. Let $z = \sigma_{A\cup\{n\}}(x) \in F_A \subseteq \Pi(S)$. 
    
    If $y_i = 0$ for some $i \in A$. Then, $y \le
    \sigma_{\{i\}}(x).$ Since we assumed $S = \Fold_{A \setminus
      \{i\}}(S)$, we know that $\sigma_{(A\setminus \{i\})\cup
      \{n\}}(z) = \sigma_{\{i\}}(x) \in \Pi(S)$. Thus, since $S$ is
    compressed, $y \in \Pi(S)$. But, $y_n = 0$, so $y \in \Pi(S)
    \setminus F_A \subseteq \Fold_A(S)$, contradiction.
    
    Otherwise, $x \in \Pi(S) \setminus F_{A}.$ Since $S$ is compressed
    and $y \le x$, we have that $y \in \Pi(S)$. Thus, since $y \not\in
    \Pi(\Fold_A(S))$, we have that $y \in F_{A}$. Thus $y_n = 1$, so
    $x_n = 1$. Let $z := \sigma_{A\cup\{n\}}(y) \not\in \Pi(S)$.
    
    Let $B \subseteq A$ be the coordinates $i \in B$ for which $x_i =
    1$. Then, $z' := \sigma_{B \cup \{n\}}(x)$. Since $x \ge y$, it
    can be checked that $z' \ge z$. Since $S$ is compressed and $z
    \not\in \Pi(S)$, we have that $z' \not\in \Pi(S)$. If $B
    \subsetneq A$, then this contradicts the fact that $z' \in
    \Pi(\Fold_B(S)) = \Pi(S)$. If $B = A$, then this contradicts the
    fact that $z' = \sigma_{A\cup\{n\}}(x) \in \Pi(S)$ because  $x \not\in F_A$.
    \end{proof}

    Now that we have defined the operators, we finish the
    proof. First, set $S' = \Flat(S)$. Now, topologically sort the
    subsets of $[n-1]$ by inclusion. For each $A \subseteq [n-1]$, in
    this topological order, apply $\Fold_A$ to $S'$. If it so happens
    that applying $\Fold_A$ causes $\Fold_B(S') \neq S'$ for some $B$
    earlier in the topological order, we go backtrack to the earliest
    such $B$.

    By Claim~\ref{claim:fold}, we know that $S'$ is still
    compressed after each operation. Note that each time $S'$ changes,
    $\mu(S' \cap J)$ strictly increases. Thus, after some finite
    number of applications of these operations, we will have a
    compressed $S'$ such that for all $A\subseteq [n-1]$, $\Fill(S') =
    \Fold_A(S') = S'$, $\mu(S') \ge \mu(S)$, and $\mu(\partial S') \le
    \mu(\partial S)$.

    Furthermore, since $S' = \Fold_{[n-1]}(S')$, we know that either
    $J \subseteq S'$ or $S' \subseteq J$. Now, take any $S'' \subseteq
    S'$ such that $\mu(S'') = \mu(S)$ while preserving the property
    that $J \subseteq S''$ or $S'' \subseteq J$. Since $S'' \subseteq
    S'$, we have that $\mu(\partial S'') \le \mu(\partial S') \le
    \mu(\partial S),$ as desired.
\end{proof}

With this claim proven, we may now prove the theorem.

\begin{proof}[Proof of Theorem~\ref{lem:bound-D-exact}.]  We divide the
proof into four parts.
  \begin{itemize}
  \item \textit{Part 1: If $\nu < \frac{1}{t}$ then $\Phi_t(\nu) \le
\frac{t-1}{t}\Phi_t(t\nu).$}

Consider any $S \subset [t]^n$ such
that $\mu(S) \ge t\nu$. Let
\[S' = S \times [1] \subset [t]^{n+1}\] be the set
where every element of $S$ has a $1$ appended. Note that
\[
\partial S' = (\partial S) \times \{2, \hdots, t\}.
\]
Therefore,
\begin{align*}
\mu(S') &= \frac{\mu(S)}{t} \ge \nu\\
\mu(\partial S') &= \frac{t-1}{t}\mu(\partial S).
\end{align*}
Thus,
\[\frac{t-1}{t}\Phi_t(t\nu) = \inf_{\substack{S\in \mathbb [t]^{*}\\\mu(S) \ge
t\nu}}\frac{t-1}{t}\mu(\partial(S)) \ge \inf_{\substack{S'\in \mathbb [t]^{*}\\\
\mu(S') \ge \nu}}\mu(\partial(S')) = \Phi_t(\nu),
\]
where $[t]^{*} := \bigcup_{n \ge 1} [t]^n.$

\item \textit{Part 2: If $\nu < \frac{1}{t}$ then $\Phi_t(\nu) \ge
\frac{t-1}{t}\Phi_t(t\nu).$}

Consider any $S \subset [t]^n$ such that
$\mu(S) \ge \nu$. If $n = 1$, then $S = \emptyset$, for which it is
trivial that $\Phi_t(0) = 0$. Thus, assume $n \ge 2$.

If $\mu(S) \ge \frac{1}{t}$, then by Theorem~\ref{thm:isoperimetry-tensor}, \[\mu(\partial S) \ge
\frac{t-1}{t} \ge \frac{t-1}{t}\Phi_t(t\nu).\]  Thus, we may assume
$\mu(S) < \frac{1}{t}$. By Claim~\ref{claim:pack1} there is $S' \in
[t]^n$ such that $\mu(S') \ge \nu$, $\mu(\partial S') \le
\mu(\partial S)$ and $S' \subseteq [t]^{n-1} \times [1]$. Let
\[S'' = (S')_1 \times [t] = \{(x_1, \hdots, x_{n-1}, y) \mymid x \in [t]^n, y \in [t]\} \subset [t]^n.\]
Intuitively, $S''$ is $S'$ `stacked' $t$ times. Therefore, $\mu(S'')
\ge t\nu$.Then,
\begin{align*}
 \partial S' &= (\partial S')_1 \times \{2, \hdots, t\}\\
 \partial S'' &= (\partial S')_1 \times \{1, \hdots, t\}.
\end{align*}

Therefore, \[\mu(\partial S) \ge \mu(\partial S') =
\frac{t-1}{t}\mu(\partial S'') \ge \frac{t-1}{t}\Phi_t(t\nu).\] Thus,
    \[ \Phi_t(\nu) = \inf_{\substack{S\in [t]^*\\\mu(S)
        \ge \nu}} \mu(\partial S) \ge  \frac{t-1}{t}\Phi_t(t\nu).
    \]

  \item \textit{Part 3: If $\nu \ge \frac{1}{t}$ then $\Phi_t(\nu) \le
\frac{t-1}{t} + \frac{1}{t}\Phi_t\left(\frac{t\nu-1}{t-1}\right).$}

For any $S \subseteq [t]^n$ such that $\mu(S) \ge \frac{t\nu -
1}{t-1}$, let $S' \subseteq [t]^{n+1}$ be
\[
S' := ([t]^n \times [1]) \cup (S \times \{2, \hdots, t\}).
\]
Then,
\[
\partial S' = ([t]^n \times \{2, \hdots, t\}) \cup (\partial S \times
[1]).
\]
Therefore,
\begin{align*}
\mu(S') &= \frac{1}{t} + \frac{t-1}{t}\mu(S)\\
\mu(\partial S') &= \frac{t-1}{t} + \frac{1}{t}\mu(\partial S).
\end{align*}
Hence,
\begin{align*} \frac{t-1}{t} + \frac{1}{t}\Phi_t\left(\frac{t\nu -
      1}{t-1}\right) = \inf_{\substack{S\in \mathbb [t]^*\\\mu(S)\ge
      \frac{t\nu-1}{t-1}}} \left(\frac{t-1}{t} +
    \frac{1}{t}\mu(\partial S)\right)\ge \inf_{\substack{S'\in
      \mathbb [t]^*\\\mu(S')\ge \nu}} \mu(\partial S') = \Phi_t(\nu).
\end{align*}

  \item \textit{Part 4: If $\nu \ge \frac{1}{t}$ then $\Phi_t(\nu) \ge
\frac{t-1}{t} + \frac{1}{t}\Phi_t\left(\frac{t\nu-1}{t-1}\right).$}

For any $S \subseteq [t]^n$ such that $\mu(S) \ge \nu \ge \frac{1}{t}$, by
Claim~\ref{claim:pack1} there is $S' \in [t]^n$ such that
$\mu(S) = \mu(S')$, $\mu(\partial S') \le \mu(\partial S)$, and
$[t]^{n-1} \times [1]\subseteq S'$. Pick $j \in \{2, \hdots, t\}$ such
that $\mu(S'_j)$ is maximal.\footnote{Note that we did not define
  $S_j$ in (\ref{eq:S_i}) for $n=1$. In that case, define $S'_j$ to be
  $\emptyset$ if $j \not\in S'$ and $\{()\}$ if $j \in S'$. It is
  consistent to define $\mu(\emptyset) = 0$ and $\mu(\{()\}) = 1$.} Then
\[
\mu(S'_j) \ge \frac{1}{t-1} \sum_{j=2}^{t}\mu(S'_j) =
\frac{1}{t-1}(t\mu(S') - \mu(S'_1)) = \frac{t\mu(S') - 1}{t-1}\mu(S')\ge \frac{t\nu - 1}{t-1}.
\]
and also
\[
\partial S' \subseteq ([t]^{n-1} \times \{2, \hdots, t\}) \cup
(\partial S'_j \times [1]).
\]
Thus,
\[
\mu(\partial S) \ge \mu(\partial S') \ge \frac{t-1}{t} + \frac{1}{t}\mu(\partial S'_j) \ge
\frac{t-1}{t} + \frac{1}{t}\Phi_t\left(\frac{t\nu - 1}{t-1}\right).
\]

Therefore,
    \[ \Phi_t(\nu) = \inf_{\substack{S\in \mathbb [t]^*\\\mu(S)\ge \nu}}
\mu(\partial S) \ge \frac{t-1}{t} +
\frac{1}{t}\Phi_t\left(\frac{t\nu - 1}{t-1}\right).\]
  \end{itemize}
\end{proof}

\section{Optimality of exponent in Theorem~\ref{thm:tensor-stability}}\label{app:opt}

In this appendix, we show in (\ref{eq:stab-tensor-clique}) of
Theorem~\ref{thm:tensor-stability} that the exponent $\eta(t) =
\frac{\log t}{\log t - \log(t-1)}$ is
optimal and that the constant factor of $2$ is nearly optimal. In
other words, the stability result is optimal up to a constant factor.

\begin{lem} For all $t \ge 3$, there exists an infinite sequence of
  independent sets $\{I_n\}_{n \ge 3}$ such that $I_n \subset \mathbb
  [t]^n$, $\epsilon_n = 1 - t\mu(I_n) > 0$ tends to $0$ as $n \to
  \infty$, and for any $n$ and any maximum-sized independent set $J_n$
  of $K_t^n$, \[
  \mu(I_n \setminus J_n) >
  \frac{t-1}{t}\epsilon^{\eta(t)}.\]
\end{lem}

\begin{proof} For $n \ge 3$, consider $J_n = [1] \times [t]^{n-1}$ and
\begin{align}
I_n := (([t] \times [1]^{n-1}) \cup J_n) \setminus ([1] \times \{2, \hdots, t-1\}^n)
\end{align}
See Figure~\ref{fig:4} for a visualization.

\begin{figure}
\begin{center}
\includegraphics[width=4in]{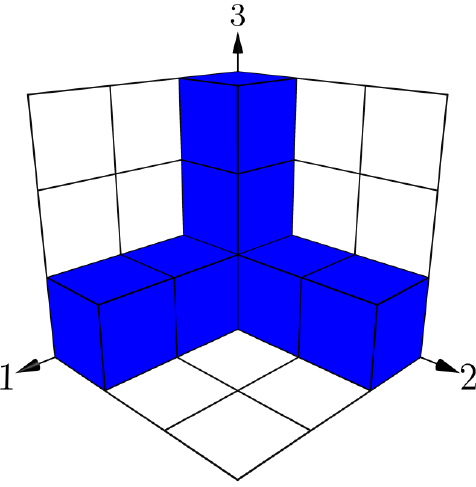}
\end{center}
\caption{Schematic of $I_3$ when $t = 3$. See the caption for Figure~\ref{fig:3_fill}
  on interpreting this visualization.}
\label{fig:4}
\end{figure}

One may check that $I_n$ is an independent set of $K_t^n$ and $J_n$ is
a maximum-sized independent set which minimizes $\mu(I_n \setminus
J_n)$. Furthermore,
\[\mu(I_n) = \frac{t-1}{t^{n}} + \frac{1}{t} - \frac{(t-1)^{n-1}}{t^{n}}.\]
Thus,
\begin{align}
\epsilon_n &= \frac{(t-1)^{n-1} - (t-1)}{t^{n-1}}\\
\delta_n &:= \mu(I_n\setminus J_n) = \frac{t-1}{t^n}.
\end{align}

Notice that since $t^{1/\eta(t)} = \frac{t-1}{t}.$
  \begin{align*}\delta_n^{1/\eta(t)} &=
  \frac{(t-1)^{1/\eta(t)}}{t^{n/\eta(t)}}\\
  &=
  \left(\frac{t-1}{t}\right)^{1/\eta(t)}\left(\frac{t-1}{t}\right)^{n-1}\\
  &= \left(\frac{t-1}{t}\right)^{1/\eta(t)}(\epsilon_n + t\delta_n)\\
  &> \left(\frac{t-1}{t}\right)^{1/\eta(t)}\epsilon_n.
\end{align*}
Therefore, raising both sides to the $\eta(t)$ power,
\[\delta_n > \frac{t-1}{t}\epsilon_n^{\eta(t)},\] as
  desired.
\end{proof}

\end{document}